\documentclass[10pt,reqno]{amsart}
\usepackage{latexsym,amsmath,amssymb,amscd, enumerate}
\usepackage[all]{xy}

\begin{document}


\makeatletter
\@addtoreset{figure}{section}
\def\thefigure{\thesection.\@arabic\c@figure}
\def\fps@figure{h,t}
\@addtoreset{table}{bsection}

\def\thetable{\thesection.\@arabic\c@table}
\def\fps@table{h, t}
\@addtoreset{equation}{section}
\def\theequation{
\arabic{equation}}
\makeatother

\newcommand{\bfi}{\bfseries\itshape}

\newtheorem{theorem}{Theorem}
\newtheorem{acknowledgment}[theorem]{Acknowledgment}
\newtheorem{corollary}[theorem]{Corollary}
\newtheorem{definition}[theorem]{Definition}
\newtheorem{example}[theorem]{Example}
\newtheorem{lemma}[theorem]{Lemma}
\newtheorem{notation}[theorem]{Notation}
\newtheorem{proposition}[theorem]{Proposition}
\newtheorem{remark}[theorem]{Remark}
\newtheorem{setting}[theorem]{Setting}
\newtheorem{hypothesis}[theorem]{Hypothesis}

\numberwithin{theorem}{section}
\numberwithin{equation}{section}

\renewcommand{\1}{{\bf 1}}
\newcommand{\Ad}{{\rm Ad}}
\newcommand{\Alg}{{\rm Alg}\,}
\newcommand{\Aut}{{\rm Aut}\,}
\newcommand{\ad}{{\rm ad}}
\newcommand{\Borel}{{\rm Borel}}
\newcommand{\botimes}{\bar{\otimes}}
\newcommand{\Ci}{{\mathcal C}^\infty}
\newcommand{\Cint}{{\mathcal C}^\infty_{\rm int}}
\newcommand{\Cpol}{{\mathcal C}^\infty_{\rm pol}}
\newcommand{\Der}{{\rm Der}\,}
\newcommand{\de}{{\rm d}}
\newcommand{\ee}{{\rm e}}
\newcommand{\End}{{\rm End}\,}
\newcommand{\ev}{{\rm ev}}
\newcommand{\hotimes}{\widehat{\otimes}}
\newcommand{\id}{{\rm id}}
\newcommand{\ie}{{\rm i}}
\newcommand{\iotaR}{\iota^{\rm R}}
\newcommand{\GL}{{\rm GL}}
\newcommand{\gl}{{{\mathfrak g}{\mathfrak l}}}
\newcommand{\Hom}{{\rm Hom}\,}
\newcommand{\Img}{{\rm Im}\,}
\newcommand{\Ind}{{\rm Ind}}
\newcommand{\Ker}{{\rm Ker}\,}
\newcommand{\Lie}{\text{\bf L}}
\newcommand{\Mt}{{{\mathcal M}_{\text t}}}
\newcommand{\m}{\text{\bf m}}
\newcommand{\pr}{{\rm pr}}
\newcommand{\Ran}{{\rm Ran}\,}
\renewcommand{\Re}{{\rm Re}\,}
\newcommand{\so}{\text{so}}
\newcommand{\spa}{{\rm span}\,}
\newcommand{\Tr}{{\rm Tr}\,}
\newcommand{\tw}{\ast_{\rm tw}}
\newcommand{\Op}{{\rm Op}}
\newcommand{\U}{{\rm U}}
\newcommand{\UCb}{{{\mathcal U}{\mathcal C}_b}}
\newcommand{\weak}{\text{weak}}

\newcommand{\CC}{{\mathbb C}}
\newcommand{\RR}{{\mathbb R}}
\newcommand{\TT}{{\mathbb T}}

\newcommand{\Ac}{{\mathcal A}}
\newcommand{\Bc}{{\mathcal B}}
\newcommand{\Cc}{{\mathcal C}}
\newcommand{\Dc}{{\mathcal D}}
\newcommand{\Ec}{{\mathcal E}}
\newcommand{\Fc}{{\mathcal F}}
\newcommand{\Hc}{{\mathcal H}}
\newcommand{\Kc}{{\mathcal K}}
\newcommand{\Jc}{{\mathcal J}}
\newcommand{\Lc}{{\mathcal L}}
\renewcommand{\Mc}{{\mathcal M}}
\newcommand{\Nc}{{\mathcal N}}
\newcommand{\Oc}{{\mathcal O}}
\newcommand{\Pc}{{\mathcal P}}
\newcommand{\Qc}{{\mathcal Q}}
\newcommand{\Sc}{{\mathcal S}}
\newcommand{\Tc}{{\mathcal T}}
\newcommand{\Vc}{{\mathcal V}}
\newcommand{\Uc}{{\mathcal U}}
\newcommand{\Xc}{{\mathcal X}}
\newcommand{\Yc}{{\mathcal Y}}
\newcommand{\Wig}{{\mathcal W}}

\newcommand{\Bg}{{\mathfrak B}}
\newcommand{\Fg}{{\mathfrak F}}
\newcommand{\Gg}{{\mathfrak G}}
\newcommand{\Ig}{{\mathfrak I}}
\newcommand{\Jg}{{\mathfrak J}}
\newcommand{\Lg}{{\mathfrak L}}
\newcommand{\Pg}{{\mathfrak P}}
\newcommand{\Sg}{{\mathfrak S}}
\newcommand{\Xg}{{\mathfrak X}}
\newcommand{\Yg}{{\mathfrak Y}}
\newcommand{\Zg}{{\mathfrak Z}}

\newcommand{\ag}{{\mathfrak a}}
\newcommand{\bg}{{\mathfrak b}}
\newcommand{\dg}{{\mathfrak d}}
\renewcommand{\gg}{{\mathfrak g}}
\newcommand{\hg}{{\mathfrak h}}
\newcommand{\kg}{{\mathfrak k}}
\newcommand{\mg}{{\mathfrak m}}
\newcommand{\n}{{\mathfrak n}}
\newcommand{\og}{{\mathfrak o}}
\newcommand{\pg}{{\mathfrak p}}
\newcommand{\sg}{{\mathfrak s}}
\newcommand{\tg}{{\mathfrak t}}
\newcommand{\ug}{{\mathfrak u}}
\newcommand{\zg}{{\mathfrak z}}

\newcommand{\ZZ}{\mathbb Z}
\newcommand{\NN}{\mathbb N}
\newcommand{\BB}{\mathbb B}
\newcommand{\HH}{\mathbb H}

\newcommand{\ep}{\varepsilon}

\newcommand{\opn}{\operatorname}
\newcommand{\slim}{\operatornamewithlimits{s-lim\,}}
\newcommand{\sgn}{\operatorname{sgn}}

\makeatletter
\title[Berezin symbols on Lie groups]{Berezin symbols on Lie groups}

\author[I. Belti\c t\u a]{Ingrid Belti\c t\u a}

\address{%
Institute of Mathematics  
``Simion Stoilow''  
of the Romanian Academy\\
P.O. Box 1-764, 
Bucharest\\
Romania}

\email{ingrid.beltita@gmail.com
}

\author[D. Belti\c t\u a]{Daniel Belti\c t\u a}

\address{%
Institute of Mathematics  
``Simion Stoilow''
of the Romanian Academy\\
P.O. Box 1-764, 
Bucharest\\
Romania}

\email{beltita@gmail.com
}


\author[B. Cahen]{Benjamin Cahen}

\address{%
Laboratoire de Math\'ematiques et Applications de Metz,\\ 
UMR 7122, Universit\'e de
Lorraine (campus de Metz) et CNRS, \\
B\^at. A, Ile du Saulcy, \\
F-57045 Metz
Cedex 1,\\ 
France}

\email{benjamin.cahen@univ-lorraine.fr
}

\thanks{The research of the first two named authors has been partially supported by 
grant of the Romanian National Authority for Scientific Research and
Innovation, CNCS--UEFISCDI, project number PN-II-RU-TE-2014-4-0370.}



\date{\today}
\makeatother

\begin{abstract}
In this paper we present a general framework for Berezin covariant symbols, 
and we discuss a few basic properties of the corresponding symbol map, 
with emphasis on its injectivity in connection with some problems 
in representation theory of nilpotent Lie groups. 
\end{abstract}

\maketitle


\section{Introduction}\label{Sect1}

Let $\Vc$ be a finite-dimensional complex Hilbert space and $N$ be a second countable smooth manifold  
with a fixed Radon measure~$\mu$.  
We denote by $L^2(N,\Vc;\mu)$ the complex Hilbert space of (equivalence classes of) 
$\Vc$-valued functions $\mu$-measurable on $N$ that are absolutely square integrable with respect to~$\mu$. 
We also endow the space of smooth functions $\Ci(N,\Vc)$ with the Fr\'echet topology of uniform convergence on compact sets 
together with their derivatives of arbitrarily high degree. 

If $\Hc\subseteq L^2(N,\Vc)$ is a closed linear subspace with $\Hc\subseteq\Ci(N,\Vc)$, 
then the inclusion map $\Hc\hookrightarrow\Ci(N,\Vc)$ is continuous, 
hence for every $x\in N$ the evaluation map $K_x\colon\Hc\to \Vc$, $f\mapsto f(x)$, is continuous. 
The map 
$$K\colon N\times N\to\Bc(\Vc),\quad K(x,y):=K_xK_y^*$$
is called the \emph{reproducing kernel} of the Hilbert space~$\Hc$.   
Then for every linear operator $A\in\Bc(\Hc)$ we define its \emph{full symbol} as 
$$K^A\colon N\times N\to\Bc(\Vc),\quad K^A(x,y):=K_xAK_y^*\colon \Vc\to \Vc$$
and  $K^A\in\Ci(N\times N,\Bc(\Vc))$.
See \cite[\S I.2]{Ne00} for a detailed discussion of this construction, which goes back to \cite{Be1} and \cite{Be2}.  

\subsection*{Main problem} 
In the above setting, the full symbol map 
$$\Bc(\Hc)\to\Ci(N\times N,\Bc(\Vc)),\quad A\mapsto K^A$$
is injective, as easily checked (see also Proposition~\ref{basic}\eqref{basic_item2} below). 
Therefore it is interesting to find sufficient conditions on a continuous map $\iota\colon \Gamma\to N\times N$, 
ensuring that the corresponding $\iota$-restricted symbol map 
$$S^\iota\colon \Bc(\Hc)\to\Cc(\Gamma,\Bc(\Vc)),\quad A\mapsto K^A\circ\iota$$
is still injective. 
The case of the diagonal embedding $\iota\colon \Gamma=N\hookrightarrow N\times N$, $x\mapsto (x,x)$, is particularly important and 
in this case the $\iota$-restricted symbol map is called the (non-normalized) \emph{Berezin covariant symbol map} 
and is denoted simply by~$S$, hence 
$$S\colon \Bc(\Hc)\to\Ci(N,\Bc(\Vc)),\quad 
(S(A))(x):=K_xAK_x^*\colon \Vc\to \Vc.$$ 
In the present paper we will discuss the above problem and we will briefly 
sketch an approach to that problem based on 
results from our forthcoming paper~\cite{BBC16}. 
This approach blends some techniques of reproducing kernels 
and some basic ideas of linear partial differential equations, 
in order to address a problem motivated by representation theory of Lie groups 
(see \cite{CaS}, \cite{CaSWC}, \cite{CaPad}, and \cite{CaRimut}). 
This problem is also related to some representations of infinite-dimensional Lie groups 
that occur in the study of magnetic fields (see \cite{BB11} and \cite{BB12}). 
Let us also mention that linear differential operators associated to reproducing kernels have been earlier used 
in the literature (see for instance \cite{BG14}).

\section{Basic properties of the Berezin covariant symbol map}\label{Sect2}

In the following we denote by $\Sg_p(\bullet)$ the Schatten ideals of compact operators on Hilbert spaces 
for $1\le p<\infty$.

\begin{proposition}\label{basic}
In the above setting, if $A\in\Bc(\Hc)$, then one has: 
\begin{enumerate}
\item\label{basic_item1} 
If $A\ge 0$, then $S(A)\ge 0$, and moreover $S(A)=0$ if and only if $A=0$. 
\item\label{basic_item2} 
For all $f\in\Hc$ and $x\in N$ one has 
$$(Af)(x)=\int\limits_N K^A(x,y)f(y)\de\mu(y). $$
\item\label{basic_item3} 
If $\{e_j\}_{j\in J}$ is an orthonormal basis of $\Hc$, then for all $x,y\in N$ one has
$$K^A(x,y)=\sum\limits_{j\in J}K_xe_j\otimes \overline{K_yA^*e_j}
=\sum\limits_{j\in J}e_j(x)\otimes \overline{(A^*e_j)(y)}\in\Bc(\Vc),$$
where for any $v,w\in\Vc$ we define their corresponding rank-one operator $v\otimes\overline{w}:=(\ \cdot \mid w)v\in\Bc(\Vc)$. 
\item\label{basic_item4} 
If $A\in\Sg_2(\Hc)$, then 
$$\Vert A\Vert^2_{\Sg_2(\Hc)}=\iint\limits_{N\times N}\Vert K^A(x,y)\Vert^2_{\Sg_2(\Vc)}\de\mu(x)\de\mu(y)$$ 
and if $A\in\Sg_1(\Hc)$, then 
$$\Tr A=\int\limits_{N}\Tr K^A(x,x)\de\mu(x).$$ 
\end{enumerate}
\end{proposition}

\begin{proof}
See \cite{BBC16} for more general versions of these assertions, 
in which in particular the Hilbert space $\Vc$ is infinite-dimensional. 
Assertion~\eqref{basic_item2} is a generalization of \cite[Ex. I.2.3(c)]{Ne00}, 
Assertion~\eqref{basic_item3} is a generalization of \cite[Prop. I.1.8(b)]{Ne00}, 
while Assertion~\eqref{basic_item4} is a generalization of \cite[Cor. A.I.12]{Ne00}. 
\end{proof}

\section{Examples of Berezin symbols and specific applications}\label{Sect3}

Here we specialize to the following setting: 
\begin{enumerate}
\item\label{set_i}  $G$  is a connected, simply connected, nilpotent Lie group 
with its Lie algebra $\gg$, whose center is denoted by~$\zg$, 
and $\gg^*$ is the linear dual space of $\gg$, 
with the corresponding duality pairing
$\langle\cdot,\cdot\rangle\colon\gg^*\times\gg\to\RR$. 
\item\label{set_iii}  $\pi\colon G\to\Bc(\Hc)$ be a unitary irreducible representation 
associated with the coadjoint orbit $\Oc\subseteq\gg^*$. 
\end{enumerate}
The group $G$ will be identified  with $\gg$ via the exponential map, so that 
$G= (\gg, \cdot_G)$, where $\cdot_G$ is the Baker-Campbell-Hausdorff multiplication. 

We use the notation $\Hc_\infty= \Hc_\infty(\pi)$ for the nuclear 
Fr\'echet space of smooth vectors of $\pi$. 
Let then 
$\Hc_{-\infty}$ be the space of antilinear continuous functionals on $\Hc_\infty$, 
$\Bc(\Hc_\infty,\Hc_{-\infty})$ be the space of continuous linear operators between the above space 
(these operators are thought of as possibly unbounded linear operators in $\Hc$), 
and $\Sc(\bullet)$ and $\Sc'(\bullet)$ for the spaces of Schwartz functions and tempered distributions, respectively.
Then we have that 
$$\Hc_\infty\hookrightarrow\Hc\hookrightarrow\Hc_{-\infty}.$$   
Let $X_1, \dots, X_m$ be a Jordan-H\"older basis in $\gg$ 
and $e\subseteq\{1,\dots,m\}$ be the set of jump indices of the coadjoint orbit $\Oc$. 
Select $\xi_0\in\Oc$ and let $\gg=\gg_{\xi_0}\dotplus\gg_e$ be its corresponding direct sum decomposition, 
where $\gg_e$ is the linear span of $\{X_j\mid j\in e\}$ 
and $\gg_{\xi_0}:=\{x\in\gg\mid[x,\gg]\subseteq\Ker\xi_0\}$. 

We need the notation for the Fourier transform. 
For $a\in \Sc(\Oc)$ we set
$$ \widehat a (x) = \int\limits_{\Oc} \ee^{-\ie \langle\xi,x\rangle} a(\xi) \de \xi, $$
where on $\Oc$ we consider the Liouville measure normalized such that the  Fourier transform
is unitary when extended to $L^2(\Oc) \to L^2(\gg_e)$. 
We denote by $\check F$ the inverse Fourier transform of $F\in L^2(\gg_0)$.

\begin{definition}\label{coeff}
\normalfont
\begin{enumerate}
\item\label{coeff_item1} 
For $f\in \Hc$ and $\phi\in \Hc$, or $f\in \Hc_{-\infty}$ and $\phi\in \Hc_\infty$, 
let $\Ac  \in C(\gg_e)\cap \Sc'(\gg_e)$ be the \emph{coefficient mapping} for $\pi$,  defined by
$$ \Ac_\phi f(x) = \Ac(f, \phi)(x):= (f\mid \pi(x) \phi),\ x\in\gg_e.$$
\item \label{coeff_item2}
For $f\in \Hc$ and $\phi\in \Hc$, or $f\in \Hc_{-\infty}$ and $\phi\in \Hc_\infty$, 
the \emph{cross-Wigner distribution}
$\Wig(f,\phi)\in\Sc'(\Oc)$  
is defined by the formula 
$$\widehat{\Wig(f,\phi)}=\Ac_\phi f.$$ 
\end{enumerate}
\end{definition}

\begin{proposition}\label{orth}
For $f, \phi\in \Hc$ we have that $\Ac(f, \phi)\in L^2(\gg_0)$, $\Wig(f, \phi)\in L^2(\Oc)$.
Moreover
\begin{align}
(\Ac(f_1, \phi_1)\mid \Ac(f_2, \phi_2))_{L^2(\gg_0)} & = (f_1\mid  f_2) \overline{(\phi_1\mid \phi_2) }
\nonumber
\\
(\Wig(f_1, \phi_1)\mid \Wig(f_2, \phi_2))_{L^2(\Oc)} & = (f_1\mid f_2) 
\overline{(\phi_1\mid \phi_2)}
\nonumber
\end{align}
for all $f_1, f_2, \phi_1, \phi_2\in \Hc$.
\end{proposition}

\begin{proof}
This follows from \cite[Prop.~2.8(i)]{BB10c}.
\end{proof}

From now on we assume  that 
\begin{equation*}
\phi\in\Hc_\infty \quad \text{with}\quad  \Vert\phi\Vert=1\quad  \text{is fixed}.
\end{equation*}
We let 
   $V\colon\Hc\to L^2(\gg_e)$  be the isometry defined by 
\begin{equation*}
  (V f)(x) := (f\mid \phi_x)\text{ for all }x\in\gg_e
  \end{equation*}
where $\phi_x:=\pi(x)\phi$. 
We denote 
\begin{equation*}
\Kc := \Ran V\subset L^2(\gg_0).
\end{equation*} 
Then $\Kc$ is a reproducing kernel Hilbert space of smooth functions, with inner product equal to the 
$L^2(\gg_0)$-inner product, so the present construction is a special instance of the general framework of Section~\ref{Sect1} 
with $\Vc=\CC$.
 
The reproducing kernel of $\Kc$ is given by 
\begin{equation*}
K(x, y) = (\pi(x) \phi\mid \pi(y) \phi)=(\phi_x\mid\phi_y), 
\end{equation*}
and  $K_y(\cdot):= K(\cdot,y) \in \Ran V$, for all $y\in \gg_0$.
We also note that 
\begin{equation*}
(\forall x\in\gg_0)\quad K_x = V \phi_x. 
\end{equation*}
The Berezin covariant symbol of an operator $T\in \Bc(\Kc)$ is then the 
bounded continuous function
$$ S(T)\colon\gg_e\to\CC,\quad S(T)(x) = (TK_x\mid K_x)_{\Kc}.$$
One thus obtains a well-defined bounded linear operator
$$S\colon \Bc(\Kc)\to\Ci(\gg_e)\cap L^\infty(\gg_e)$$ 
which also gives by restriction a bounded linear operator 
$$S\colon \Sg_2(\Kc) \to L^2(\gg_0).$$
To find accurate descriptions of the kernels of the above operators is a very important problem 
for many reasons, as explained in \cite{CaS}, \cite{CaSWC}, \cite{CaPad}, and \cite{CaRimut} 
also for other classes of Lie groups than the nilpotent ones.  

\subsection*{The case of flat coadjoint orbits of nilpotent Lie groups}

We now assume that the coadjoint orbit $\Oc$ is flat, 
hence its corresponding representation $\pi$ is square integrable modulo the center of~$G$. 

\begin{remark}\label{rho}
\normalfont 
Consider the representation $\rho\colon G \to \Bc (\Kc)$, 
$$ \rho(g) = V \pi(g) V^*, $$
that is a unitary representation of $G$ equivalent to $\pi$, thus it corresponds to the same coadjoint orbit $\Oc$. 
We denote by $\Op_\rho$ the Weyl calculus corresponding to this representation. 
The following then holds:
\begin{enumerate}
\item \label{item_rho1}
For  $a\in \Sc'(\Oc)$ one has $ \Op_\rho(a)= V \Op(a)V^* = T_a$. 
\item \label{item_rho2}
 For $T\in \Bc(\Kc)$ and $X\in \gg_0$, one has 
 \begin{equation}\label{cov1}
  S(\rho(x)^{-1} T \rho(x))(z)= S(T)(x\cdot z), \qquad \text{for all} \; z\in\gg_0.
  \end{equation} 
\end{enumerate} 
\end{remark}

\begin{theorem}\label{Inj_S}
Assume that in the constructions above, 
\begin{equation}\label{left_cyclic}
\phi\in \Hc_\infty \quad \text{ is such that $\Wig(\phi, \phi)$ 
is  a cyclic vector for $\alpha$.}
\end{equation}
Then $S\colon \Sg_2(\Kc) \to L^2(\gg_0)$ is injective. 
\end{theorem}

\begin{proof}
The method of proof is based on specific properties of the Weyl-Pedersen calculus 
from \cite{BB10c}. 
\end{proof}

We refer to \cite{BBC16} for a more complete discussion  and for 
proofs of the above assertions in a much more general setting.  
To conclude this paper we will just briefly discuss an important example. 

\subsection*{The special case of the Heisenberg groups}

Let $G$ be the Heisenberg group  
of dimension $2n+1$ and $H$ be the center of $G$. 
Let
$\{X_1,\ldots,X_n,Y_1,\ldots,Y_n, {Z}\}$ be a basis of
${\mathfrak g}$ in which the only non trivial brackets are
$[X_k\,,\,Y_k]= {Z}$, $ 1\leq k\leq n$ and let
$\{X_1^{\ast},\ldots,X_n^{\ast},Y_1^{\ast},\ldots,Y_n^{\ast},{
Z}^{\ast}\}$ be the corresponding dual basis of ${\mathfrak
g}^{\ast}$.

For $a=(a_1,a_2,\ldots,a_n)\in {\mathbb R}^n$, $b=(b_1,b_2,\ldots,b_n)\in
{\mathbb R}^n$ and $c\in {\mathbb R}$, we denote by $[a,b,c]$ the
element $\exp_{G}(\sum_{k=1}^na_kX_k+\sum_{k=1}^nb_kY_k+c{
Z})$ of $G$. Then the multiplication of $G$ is given by
\begin{equation*}[a,b,c][a',b',c']=[a+a',b+b',c+c'+\frac{1}{2}(ab'-a'b)]\end{equation*}
and $H$ consists of all elements of the form $[0,0,c]$ with $c\in {\mathbb R}$.

The coadjoint action of $G$ is then given by
\begin{equation*}
\begin{aligned}
\Ad ^{\ast} & ([a,b,c])\, \left(
\sum_{k=1}^n{\alpha}_kX_k^{\ast}+\sum_{k=1}^n{\beta}_kY_k^{\ast}+
{\gamma}{ Z}^{\ast}\right) \\
&=\sum_{k=1}^n({\alpha_k}+{\gamma}
{b_k})X_k^{\ast}+ \sum_{k=1}^n({\beta}_k-{\gamma}{ a_k})Y_k^{\ast}+
{\gamma}{ Z}^{\ast}. 
\end{aligned}
\end{equation*}

Fix a real number $\lambda>0$. By the Stone-von Neumann
theorem, there exists a unique (up to unitary equivalence) unitary
irreducible representation $\pi_0$ of $G$ whose restriction to $H$ is the character  $\chi:[0,0,c]\rightarrow e^{i\lambda c}$.  
This representation is realized on
${\mathcal H}_0=L^2({\mathbb R}^n)$ as
\begin{equation*}\pi_{0}([a,b,c])(f)(x)
=e^{i\lambda(c-bx+\frac{1}{2}ab)}f(x-a). 
\end{equation*}
Here we take $\phi$ to be the function $\phi(x)=\left(\frac{\lambda}{\pi}\right)^{1/4}e^{-\lambda x^2/2}$. 
Then we have $\Vert \phi \Vert_2=1$. 

Theorem~\ref{Inj_S} gives a new proof of the following known fact: 

\begin{corollary} The map $S$ is a bounded linear operator from
$\Sg_2({\mathcal H}_0)$  to $L^2(\RR^{2n})$  which
is one-to-one and has dense range. 
\end{corollary}



\end{document}